\documentclass{article}[10]
\usepackage{amssymb,amsmath,mathdots,amsthm}
\usepackage{epsfig}

\newcommand{\eq}[1]{(\ref{#1})}
\newcommand{\la}{\lambda}
\newcommand{\La}{\Lambda}

\newcommand{\eps}{\varepsilon}
\newtheorem{Assumptiona}{Assumption}
\newtheorem{Definitiona}{Definition}
\newtheorem{Theorema}{Theorem}

\newcommand{\CC}{\mathbb{C}}
\newcommand{\II}{\mathbb{I}}

\newcommand{\RR}{\mathbb{R}}
\newcommand{\LL}{\mathcal L}
\newcommand{\De}{\mathcal D}
\newcommand{\UU}{\mathcal U}
\newcommand{\VV}{\mathcal V}
\newcommand{\Lin}{\tilde{\mathcal{L}}^{-1}}

\newcommand{\spanv}{\mathop{\mathrm{span}}} 
\newcommand{\gKer}{\mathop{\mathrm{gKer}}}
\newcommand{\Ker}{\mathop{\mathrm{Ker}}} 

\newcommand{\Ran}{\mathop{\mathrm{Ran}}}
\renewcommand{\Re}{\mathop{\rm Re}\nolimits}
\renewcommand{\Im}{\mathop{\rm Im}\nolimits}
\newcommand{\sign}{\mathop{\rm sign}\nolimits}

\newcommand{\nun}{n_{\text{\rm uns}}}

\newcommand{\hKer}{\mathop{\widehat{\mathrm{Ker}}}} 
\newcommand{\Haragus}{H{\u a}r{\u a}gu\c{s}}
\newcommand{\ml}{l\kern-0.035cm\char39\kern-0.03cm}
\newcounter{example}
\newenvironment{example}[1][]{\refstepcounter{example}\par\medskip\noindent%
 $\vartriangleleft$  \textit{Example~\theexample. #1} \rmfamily}{$\vartriangleright$ \medskip}

\begin{document}

\title{Index Theorems for Polynomial Pencils}
\author{Richard Koll\'ar, Radom{\'\i}r  Bos{\'a}k \\
{\small  Department of Applied Mathematics and Statistics} \\
{\small Faculty of Mathematics, Physics, Informatics, Comenius University} \\
{\small Mlynsk{\'a} dolina, 842 48 Bratislava, Slovakia} \\
{\small E-mail: kollar@fmph.uniba.sk}}
\date{\today}


%

\maketitle

\begin{abstract}
We survey index theorems counting eigenvalues of linearized Hamiltonian systems
and characteristic values of polynomial operator pencils. We present a simple common graphical interpretationand generalization  of the index theory using the concept of graphical Krein signature.
Furthermore, we prove that derivatives of an eigenvector $u= u(\la)$ of an operator pencil $\LL(\la)$ satisfying $\LL(\la) u(\la)= \mu(\la)  u(\la)$ evaluated at a characteristic value of $\LL(\la)$ do not only generate an arbitrary chain of root vectors of $\LL(\la)$ but the chain that carries an extra information.   
\end{abstract}

\section{Introduction}\label{sec:RKintro}

\emph{Spectral problems} naturally arise in investigations of 
stability and decay rates of nonlinear waves, in stability analysis of numerical schemes, in integrable systems solved via
inverse scattering method, and in multiple other fields.  The main interest is in presence of
\emph{unstable point spectrum}---points  in the point spectrum of the particular operator with a positive real part
that correspond to destabilizing modes. 
Over the last 40 years counts of unstable point spectra and other related counts that we commonly refer to as  \emph{index theorems}
appeared across  various distinct and unrelated fields due to their simple structure and importance for applications. 
Here we briefly survey literature on index theory, point out its common graphical interpretation, and derive 
its generalization to problems with operators with arbitrary structure of the kernel.

\emph{Linearized Hamiltonian systems.}
Index theorems proved to be particularly useful in spectral stability theory of waves in Hamiltonian systems where one studies the spectrum $\sigma(J\!L)$ of the non-selfadjoint problem
\begin{equation}
J\! L u = \nu u, \qquad\qquad  J = -J^{\ast}, \ L = L^{\ast},
\label{JL}
\end{equation}
where $J$ and $L$ are operators acting on a Hilbert space $X$, $L$ is the second variation Hessian of the underlying Hamiltonian, 
$L^{\ast}$ denotes the adjoint operator of~$L$, and $u \in X$ \cite{Swaters, Meiss}.
Problem \eq{JL} appears in search for exponentially growing or decaying solutions  $v(x,t) = e^{\nu t} u(x)$ 
of the linearized system $v_t = J\!L v$  obtained  by linearization of the Hamiltonian system around its equilibrium. 
Here $v(x,t)$ represents an infinitesimal perturbation of the equilibrium that is 
said to be spectrally stable if  \eq{JL} has no solution with $\Re \nu > 0$ for $u \in X$. 
Due to the natural symmetry of spectrum $\sigma(J\! L)$ (see \cite{KM}) the spectral stability is equivalent to the confinement 
of $\sigma(J\!L)$ to the imaginary axis. While positivity of the spectrum of $L$ implies spectral stability, 
the operator $L$ often has negative eigenvalues, and due to the symmetries of the system also a non-trivial kernel allowing 
an instability in the system. However, the symmetries through the Noether theorem
imply existence of conserved quantities (typically corresponding to physically meaningful quantities as mass, momentum, etc.). 
Their conservation restricts possible degrees of freedom in the system and thus can prohibit instability in cases of 
indefinite $L$. The index theorems can be applied in such situations as they relate the number of negative real eigenvalues of $L$ and the count of unstable spectra of $J\!L$. On the other hand, analogous index theorems proved to be useful in the quadratic operator pencils setting. The link between these two types of results is that for invertible $J$ it is possible to reformulate \eq{JL} 
as a linear operator pencil. Therefore, both types of index theorems can be viewed as special cases of a general theory for operator pencils. 

Two different ways to interpret the index theorems  mathematically can be traced in the literature. 
Motivated by the work of Hestenes \cite{Hestenes} (see also \cite{Gregory}) 
Maddocks \cite{Maddocks1988, MS1995} derived the dimension counts for finite-dimensional restricted quadratic forms 
and showed how the question of stability of an equilibrium of a Hamiltonian system reduces to 
a question whether a quadratic form is positive when restricted to a particular subspace of its domain.
Such an approach was later used in works \cite{GSS1,  Pel2005,KKS, CP, KapHar, BeTr2007} 
and it is closely related to the theory of indefinite inner product spaces \cite{Bognar, Iohvidov,Langer}. 
It provides a geometric visualization of the index theorems as counts of the dimension of the intersection of the
negative energy cone associated with the indefinite quadratic form with the subspace spanned by normal vectors
 (under the indefinite inner product) to hyperplanes tangential to surfaces of conserved quantities (Fig.~\ref{fig1}, left panel). 

However, here we focus on an alternative viewpoint of a different geometrical (graphical) nature
\cite{BinBrown1988, KollarH,KM}. We interpret the index theorems as topological counts of curves of eigenvalues of operator pencils 
in a plane (Fig.~\ref{fig1}, right panel). We believe that such an interpretation provides besides the simpler visualization of the theory 
also an easier way for generalizations. Additionally, as we will show, 
it also yields reduced algebraic formulae for calculation of the indices of operators with complicated generalized kernels
due to the fact that chains of root vectors generated by the graphical method carry an extra information compared to 
(generic) chains of root vectors.

\begin{figure}[t]
\includegraphics[scale=0.32]{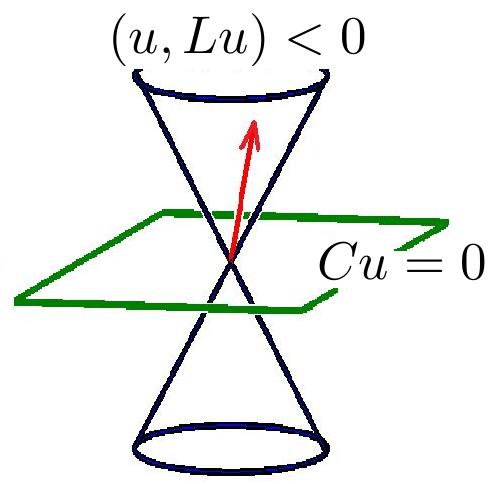}
\hspace{1cm}
\includegraphics[scale=0.3]{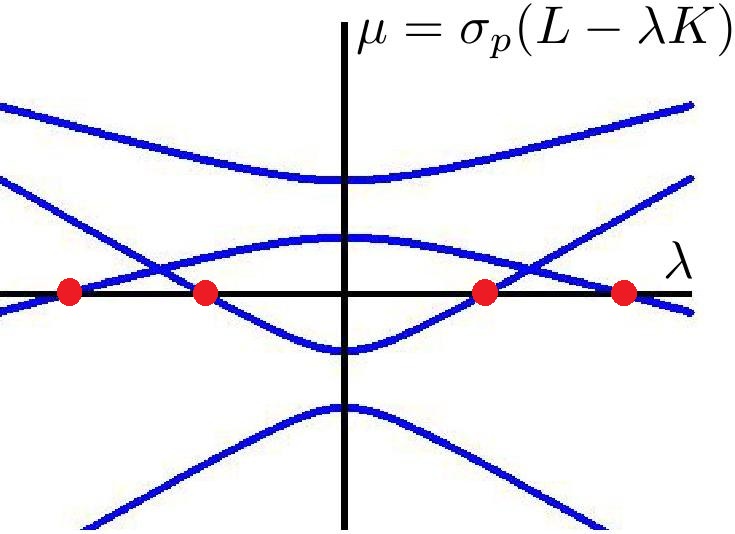}
\caption{\small{Visualization of the index theory. 
Left panel:  Algebraic approach. The equilibrium is (spectrally) stable if and only if  the normal vector
(in the associated indefinite inner product space) to the plane $Cu = 0$ corresponding to the invariant (conserved quantity) of the system lies in the negative energy cone ${\mathcal C} = \{u\in X, (u,Lu)<0\}$ \cite{Maddocks1988}.\newline
Right panel: Graphical approach. Eigenvalue branches $\mu = \mu(\la)$ (the point spectrum $\sigma_p(L-\la K)$) 
of the eigenvalue pencil $(L - \la K)u = \mu u$ are plotted vs.~$\lambda$.  
Purely imaginary eigenvalues $\nu$ of $J\! L$ correspond to intercepts of $\mu(\la)$ with the axis $\mu = 0$ via 
$\la = i \nu$ (indicated by full circles). Their Krein signature is given by the sign of $\mu'(\la)$ \cite{KM}.
\label{fig1}}}
\end{figure}

This work can be viewed as an extension of the theory developed by Koll{\'a}r and Miller \cite{KM} who laid down the groundwork 
 for the analysis and derived special cases of theorems presented here. Although both the present paper and \cite{KM}  
survey the literature on index theorems, these surveys are quite different. While \cite{KM} 
is exclusively focused on the literature appearing in the field of stability of nonlinear waves, 
here we use the opportunity to follow the idea of the BIRS workshop and bridge various different
fields of theoretical and applied mathematics where the results appeared parallely over the  years. 
The lack of such survey gathering results from various fields served as a motivation for Section~\ref{s:IT}.  We hope that our work will contribute to increased communication  and thus faster transfer of results between various fields in the future.

\section{Krein Signature}\label{sec:KS}

Index theorems often refer to \emph{Krein signature} of a characteristic value, a quantity that characterizes ability 
of the (stable) characteristic value to become unstable under a perturbation \cite{Krein1950}. In linearized Hamiltonian systems
the Krein signature $\kappa_L(\la)$ of a characteristic value $\nu$ of $J\!L$ 
captures the signature of the quadratic form  $(\cdot, L\, \cdot)$
representing the linearized energy on the invariant subspace spanned by root spaces corresponding to $(\nu, -\overline{\nu})$
(see MacKay \cite{MacKay} for the geometric visualization). Krein signature is also referred to as \emph{sign characteristics} \cite{GLR} within the context of operator pencils and \emph{symplectic signature} in Hamiltonian mechanics (see Kirillov \cite{Kirillov2009} for a detailed discussion of the terminology, literature survey, and extension of the results of \cite{MacKay}). If the signature of the quadratic form on the subspace is indefinite, the Krein signature is said to be \emph{indefinite}, 
otherwise it is \emph{definite (positive or negative)}. It is easy to see that the Krein signature of non-purely imaginary 
and non-semi-simple purely imaginary characteristic values is indefinite \cite{Pontryagin, Iohvidov}. 
On the other hand, the signature of any simple non-zero purely imaginary characteristic value of \eq{JL} is definite \cite{KP}.

If $J$ is invertible it is useful to define $K = (iJ)^{-1}$. 
Then  $Lu = i\nu K u$  for a simple purely imaginary characteristic value $\nu$ of $J\! L$ with the characteristic vector $u$ and 
\begin{equation*}
\kappa_L(\nu)  := \sign (Lu, u) = \sign i\nu (u, Ku)\, .
\end{equation*}
Thus the sign of  $(u, Ku)$ agrees with $\kappa_L(\nu)$ up to the sign of 
$i \nu$ and one can define
\begin{equation}
\kappa_K(\la) :=  -(u, Ku)\, , \qquad \mbox{for $\la := i\nu$,\quad $Lu = \la K u$.}
\label{kksig}
\end{equation} 
Since the definition \eq{kksig} is closely related to its graphical analogue, we will drop the index $K$ in \eq{kksig} in the rest of the paper where confusion will not arise.
Due to the rotation $\la = i\nu$ the main interest lies in Krein signature of $\la \in \RR$, so we limit ourselves to a definition of Krein signature of real  characteristic values of \eq{kksig}.  See \cite{KM} for the detailed definitions of a characteristic value of an operator pencil, its geometric and algebraic multiplicities, the definitions of the maximal chain of root vectors and 
the canonical set of maximal chains can be also found in \cite{GLR}. 

\begin{Definitiona}
Let $J$ be an invertible skew-adjoint and $L$ a self-adjoint operator on the Hilbert space $X$. 
Let $\la_0$   be a real characteristic value of
$iJ\!L$ and let $\mathcal U$ be one of its maximal chains of root vectors. 
Furthermore, let $K = (iJ)^{-1}$ and 
let $W$ be the  (Hermitian) Gram matrix  of the quadratic form $(\cdot, - K\, \cdot)$ on the span of ${\mathcal U}$.  
The number of positive (negative) eigenvalues of $W$ is called the positive (negative) Krein index 
of $\UU$ at $\la_0$ and is denoted $\kappa^+(\mathcal{U},\la_0)$ ($\kappa^-(\mathcal{U},\la_0)$).  
The sums of $\kappa^\pm(\mathcal{U},\la_0)$ over the canonical set of maximal chains of root vectors $\mathcal{U}$ 
of $\la_0$ are called the positive and negative Krein indices of $\la_0$ and are denoted $\kappa^\pm(\la_0)$.  Finally, 
$\kappa(\mathcal{U},\la_0):=\kappa^+(\mathcal{U},\la_0)-\kappa^-(\mathcal{U},\la_0)$ is called
the Krein signature of the maximal chain $\mathcal{U}$ for $\la_0$, and $\kappa(\la_0):=\kappa^+(\la_0)-\kappa^-(\la_0)$ 
is called the Krein signature of $\la_0$.  
\label{d:Ksig}
\end{Definitiona}
See \cite{KM, GLR} for the proper 
analogous definition of Krein indices and Krein signature of a real characteristic value of a self-adjoint operator pencil 
of various types (Hermitian matrix pencils, compact perturbations of identity, holomorphic families of type (A) \cite{Kato}, etc.).

Now we consider the spectrum of the operator pencil $\LL(\la) := L - \la K$, i.e., the set of $\mu = \mu(\la)$ 
for which there exists $u \in X$ such that 
\begin{equation}
\LL(\la) u = \mu u\, .
\label{Lmu}
\end{equation}
There is a natural one-to-one correspondence  (including partial multiplicities) of 
the real point spectrum of $iJ\!L$ and the set of real characteristic values of $\LL(\la)$, i.e., the set of $\la_0\in \RR$ such that $\LL(\la_0)$ has a non-trivial kernel \cite{GLR, Markus, KM}. 
Under suitable assumptions the eigenvalues $\mu(\la)$ and eigenvectors $u(\la)$ can be chosen to be real analytic in $\la$ \cite{KM} and 
it is possible to define the graphical Krein signature for a self-adjoint operator pencil \cite{KM, Markus}. 

\begin{Definitiona}
Let $\LL(\la)$ be a self-adjoint operator pencil. Assume that $\LL$ has an isolated real characteristic value $\la_0$ and there are real analytic eigenvalue branches $\mu(\la)$ of $\LL(\la)$ such that eigenvalues of  $\LL(\la)$ for $\la$ close to $\la_0$ are identical to $\mu(\la)$.  
Let $\mu=\mu(\la)$ be one of the branches with $\mu^{(n)}(\la_0)=0$ for $n=0,1,\dots,m-1$, and $\mu^{(m)}(\la_0)\neq 0$.   
Let $\eta(\mu):=\mathrm{sign}(\mu^{(m)}(\la_0))=\pm 1$.
Then the 
quantities 
\begin{equation}
\kappa^\pm_\mathrm{g}(\mu,\la_0):=\begin{cases}
\tfrac{1}{2}m,&\quad \text{for $m$ even,}\\
\tfrac{1}{2}(m\pm\eta(\mu)),&\quad\text{for $m$ odd,}
\end{cases}
\end{equation}
are called the positive and negative graphical Krein indices of the eigenvalue branch $\mu=\mu(\la)$ at  $\la_0$.  
The sums of $\kappa^\pm_\mathrm{g}(\mu,\la_0)$ over all eigenvalue branches crossing at $(\la,\mu)=(\la_0,0)$ are called the positive and negative graphical 
Krein indices of $\la_0$ and are denoted $\kappa^\pm_\mathrm{g}(\la_0)$. Finally, 
$\kappa_\mathrm{g}(\mu,\la_0):=\kappa_\mathrm{g}^+(\mu,\la_0)-\kappa_\mathrm{g}^-(\mu,\la_0)$ is called the graphical Krein signature 
of the eigenvalue branch $\mu=\mu(\la)$ vanishing at $\la_0$,
and $\kappa_\mathrm{g}(\la_0):=\kappa_\mathrm{g}^+(\la_0)-\kappa_\mathrm{g}^-(\la_0)$ is called the graphical Krein signature of $\la_0$.
\label{d:GKsig}
\end{Definitiona}

Definition~\ref{d:GKsig} extends to general self-adjoint operator pencils as long as smooth eigenvalue and eigenvector branches exist in a neighborhood of an isolated characteristic value $\la_0$ \cite{KM}. 
The fundamental  relation between the Krein signature and the graphical Krein signature of a real $\la_0$ \cite{KM, Markus, BinBrown1988}
is given by
\begin{equation}
\kappa_\mathrm{g}(\la_0) = \kappa_K(\la_0):= \kappa(\la_0)\, .
\label{kreinequiv}
\end{equation}  
The Krein signature $\kappa(\la_0)$ of a characteristic value $\la_0$ of $\LL = \LL(\la)$ then can be read off the graph
of spectrum of $\LL(\la)$ in the vicinity of $\la = \la_0$ and a maximal chain of root vectors of $iJ\!L$ at $\la_0$ can be generated by derivatives of the eigenfunction branch $u(\la)$ corresponding to the eigenvalue $\mu(\la)$ of \eq{Lmu} at $\la = \la_0$  \cite{Markus, KM, KollarH}. 
The relation \eq{kreinequiv} was rigorously established for Hermitian matrix pencils in \cite{GLR} 
and for self-adjoint holomorphic operator pencils of type (A) in \cite{KM}. 

\section{Index Theorems for Linear Pencils and Linearized Hamiltonians}\label{s:IT}

Let $X$ and $Y$  be separable Hilbert spaces and let $A$ be a densely defined 
operator $D(A) \subset X \rightarrow Y$. We denote $\sigma_p(A)$ the point spectrum of $A$ and $\nun(A)$  the unstable index of $A$ counting the number of points in $\sigma_p(A) \cap \{ \mbox{Re} (z) > 0\}$. 
Furthermore,  let $p(A)$, $z(A)$, and $n(A)$ be, respectively, the counts of positive, zero, 
and negative real points in  $\sigma_p(A)$ (counting multiplicity).

In 1972  Vakhitov and Kolokolov \cite{VK} studied stability of stationary  (in an appropriate reference frame)  solutions $\phi_\omega$ of a nonlinear Schr{\"o}dinger equation parameterized by angular velocity $\omega$. Their linear stability is characterized by the spectrum of the eigenvalue problem \eq{JL}. The \emph{Vakhitov-Kolokolov criterion} states that if 
$L_\pm$ are self-adjoint operators, $L_+$ is positive definite, $L_-$ has exactly one negative eigenvalue,
\begin{equation}
J = \left( \ \, \begin{matrix}  0 & - \II \\ \II & 0 \end{matrix} \right)\, ,\qquad
L = \left( \ \,  \begin{matrix}  L_+ & 0 \\ 0 & L_- \end{matrix} \right)\, ,
\label{canstr}
\end{equation}
then  
\begin{equation}
\nun(J\!L)  =  n(L) - n(dI/d \omega)  \, .
\label{VK1}
\end{equation}
 Here $n(L) = 1$ and $I(\omega) = \int \phi_{\omega}^2 \, dx$ is the charge (momentum) of the stationary solution, i.e., $dI/d\omega$ is a scalar. 
The quantity $dI/d\omega$ can be related to the sign of the derivative $D'(\la)$ at $\la=0$ of the Evans function \cite{PegoWeinstein} and to the quadratic form $(\cdot, L\,\cdot)$ evaluated at the first generalized characteristic vector of $J\!L$ associated with 
the root vector $\phi_{\omega}$.
Thus the count of unstable point spectra of $J\!L$ depends on the number related to the properties of the generalized kernel of $J\!L$ given by $\gKer (J\!L):= \spanv_{k\ge 0} \Ker \left((J\!L)^k\right)$.

Under the assumption of the full Hamiltonian symmetry $\{\la, -\la, \overline{\la}, -\overline{\la}\}$ of $\sigma_p(J\! L)$ 
it is straightforward to generalize \eq{VK1} to the parity index theorem:  $\nun(J\!L) - (n(L) - n(D))$ is an even non-positive integer. 
Here $n(D)$ is the count of negative eigenvalues of 
a matrix $D$ related to $\gKer(J\!L)$ (see Theorem~\ref{th:KKS}).  This parity index theorem was proved 
in 1987 in celebrated papers \cite{GSS1, GSS2} by Grillakis, Shatah, and Strauss who also proved the connection of spectral stability to nonlinear stability for a wide class of problems
(see also \cite[Chapter 4]{Sulem2} and \cite{Pel2012} for the survey of the stability results in the context of nonlinear Schr{\"o}dinger and KdV equations). The generalization of the parity index theorem for the operators related to stability of dispersive waves was proved by Lin \cite{Lin2008}. The parity index theorem plays also an important role in the theory of gyroscopic stabilization where it states that if the degree of instability (negative index) of the system is odd the equilibrium point cannot be stabilized by gyroscopic forces. This fundamental theorem is often referred to as Thomson theorem \cite{Thomson, Chetaev, KozKar2004}, Thomson-Tait-Chetaev theorem \cite[Chapter 6]{Merkin} and Kelvin's theorem \cite{Kozlov2009} (Lord Kelvin's original name was William Thomson). See the works of  Kozlov \cite{Kozlov1993, Kozlov2010}, Kozlov and Karapetyan \cite{KozKar2004}, and Chern \cite{Chern2002} for further results, applications, and references,  and Kozlov \cite{Kozlov2009} for the the topological implications of the theorem.

It is easy to identify $\nun(J\!L) = k_r + 2k_c$, where $k_r$ is the number of positive 
real points in $\sigma_p(J\! L)$ and $k_c$ is the count of points in 
$\sigma_p(J\! L) \cap \{ z\in \CC, \Re(z) >0, \Im (z) > 0\}$. 
Furthermore, denote
$$
k_i^{-} = \sum_{\nu \in \sigma_p(J\!L) \cap i\RR, i\nu < 0} \kappa_L^-(\nu)
= \sum_{\la \in \sigma_p(iJ\!L) \cap \RR^-} \kappa_K^- (\la)\, .
$$
The quantity  $k_i^-$ counts the total negative Krein index of points in $\sigma_p(J\!L)\cap i\RR^+$. 
The final form of the index theorem for linearized Hamiltonians was proved in 2004 independently by 
Kapitula, Kevrekidis and Sandstede \cite{KKS} and by Pelinovsky \cite{Pel2005} (some assumptions of \cite{Pel2005} were removed in \cite{VP2006}).

\begin{Theorema}[\cite{BinBrown1988,KKS, Pel2005}]
Let $J$ be an invertible skew-adjoint 
and  $L$ a self-adjoint operator acting  on a Hilbert space $X$, with $J^{-1}$ bounded on a subspace of $X$ of a finite codimension, 
$n(L) < \infty$, and $\sigma(L) \cap \{ x\in \RR, x \le 0\} \subset \sigma_p(L)$. 
Assume that the operators $J$ and $L$ satisfy (symmetry) assumptions that imply
the full Hamiltonian symmetry of $\sigma_p(J\!L)$. Also assume that
all Jordan chains corresponding to kernel of $J\!L$ have length two.
Let  $\VV = \gKer (J\!L) \ominus \Ker (L)$
and let $D$ be the (symmetric) matrix of the quadratic form $(\cdot, L \, \cdot)$ restricted to $\VV$. 
Then 
\begin{equation}
\nun(J\!L)  = k_r + 2k_c  = n(L) - n(D) - 2k_i^-\, .
\label{nLD}
\end{equation}
\label{th:KKS}
\end{Theorema}
Note that the count \eq{nLD} was already established in 1988 by 
Binding and Browne  \cite[Proposition 5.5]{BinBrown1988}
(although the $n(D)$ term is calculated a different way when $z(L)>0$).
They considered  the case of $L$ semi-bounded with compact resolvent and $J$ 1-to-1 
and used the standard perturbation theory combined with the graphical Krein signature 
referred to as  \emph{two parameter spectral theory}. 
Their beautiful short and simple argument is based on graphical inspection of eigenvalues branches (eigencurves) that can 
be interpreted as a homotopy in the parameter $\mu_0$ from $L + \mu_0 \II$ 
positive definite to $L$ indefinite and counting of the eigenvalue branches intersections with 
the axis $\mu = 0$ (see also \cite{BinVol1996} where applications in Sturm-Liouville theory are studied). Furthermore, the claim and the proof of Theorem~\ref{th:KKS} is in some extent implicitly present in the theory developed by Iohvidov \cite{Iohvidov}, Langer \cite{Langer}, and also in Bognar \cite[Section XI.4]{Bognar}.

For the operators arising in the spectral theory of Sturm-Liouville problems with indefinite weight 
an index theorem \eq{nLD} plays an important role. While the upper bound of $\nun(J\!L)$ is well understood 
\cite[Theorem 5.8.2]{Zettl}, the exact count and dependence of its individual factors on the coefficients of the underlying differential equation poses an important open problem \cite[Problems IX--X, p.~300, Problem 1, p.~124, Comment (7), p.~128]{Zettl}.

Koll{\'a}r and Miller \cite{KM} gave a short graphical proof of the index theorem \eq{nLD} for Hermitian matrix pencils. 
We further generalize their results in Section~\ref{s:KM} and derive the generalization of a finite-dimensional version of Theorem~\ref{th:KKS}. 
Kapitula and \Haragus~\cite{KapHar} proved the analogue of \eq{nLD} for periodic Hamiltonian systems
using the Floquet theory (Bloch wave decomposition) 
under technical assumptions related to the Keldysh theorem that guarantees completeness of the eigenvectors for \eq{JL}. 
Some of the technical assumptions of \cite{KapHar} were later removed in \cite{DK2010}. 
See also \cite{BJK} for an alternative proof of \eq{nLD} based on the integrable structure of the underlying problem. 
Recently, Stefanov and Kapitula \cite{StefKap} and Pelinovsky \cite{Pel2012} removed the assumption  of boundedness of $J^{-1}$ and proved \eq{nLD} 
for the case covering the Korteweg-de Vries-type problems with $J = \partial_x$  under the assumption $\dim (\Ker L) = 1$ (see \cite{Pel2012} for historical discussion of the  stability results). 
Chugunova and Pelinovsky \cite{CP} studied the generalized eigenvalue problem $Lu = \la Ku$ using the theory 
of indefinite inner product spaces and particularly Pontryagin Invariant Subspace Theorem and proved counts (inertia laws) 
analogous to \eq{nLD}. Furthermore they showed how can be \eq{JL} treated within that context and provided an alternative proof of \eq{nLD}. 

In 1988 Jones \cite{Jones1988} and Grillakis \cite{Grillakis1988} independently proved the index theorem bounding the number of unstable points in $\sigma_p(J\!L) \cap \RR^+$ from below for the systems with the canonical form \eq{canstr}. 

\begin{Theorema}[\cite{Jones1988,Grillakis1988}]
Let $J$ and $L$ have the canonical structure \eq{canstr} with $L_\pm$ self-adjoint on a Hilbert space $X$,  $\Ker(L_+) \perp \Ker (L_-)$, and let $V$ denote the orthogonal complement of $\Ker(L_+) \oplus \Ker (L_-)$ in $X$ with the orthogonal projection $P: Y \rightarrow V$. 
Then 
\begin{equation}
\nun(J\!L) \ge k_r \ge  |n(PL_+ P) - n(PL_-P)|\, . 
\label{eq:J}
\end{equation} 
\label{th:J}
\end{Theorema}
The proofs in \cite{Grillakis1988} and in \cite{Jones1988} are significantly different, with the method 
of Grillakis \cite{Grillakis1988} related to the graphical Krein signature. 
Note that Theorem~\ref{th:J} does not rely on completeness of the root vectors of $J\!L$. 
Theorem~\ref{th:J} is frequently used to establish instability of various nonlinear waves, particularly in situations when the negative spectrum and the kernel of $L_{\pm}$ are explicitly known.

Kapitula and Promislow  \cite{KapProm} reproved Theorem~\ref{th:KKS} using the theory of \cite{Maddocks1988} 
for constrained Hamiltonian systems and the Krein matrix theory, and reformulated \eq{JL}, \eq{canstr}  
by inverting the operator $L^+$ reducing \eq{JL} to a generalized eigenvalue problem for which they established \eq{eq:J}. They also proved a local count theorem analogous to Theorem~\ref{thm:local} of Section~\ref{s:KM}.
Note that Theorem~\ref{th:J} can be also easily obtained as a corollary of a general result of \cite{KM} 
(see Section~\ref{s:KM}) by using the same reformulation as in \cite{KapProm}. 
Both counts \eq{nLD} and \eq{eq:J} were in a more general context also derived by Cuccagna {\it et al.} 
in \cite{CPV2005} in the setup allowing the point spectrum to be embedded in the essential spectrum under some further technical assumptions. 
A lower bound for the number of real eigenvalues for Hermitian matrix pencils was derived 
by Lancaster and Tismenetsky \cite{LT1983} 
together with various other index theorems for perturbed Hermitian matrix pencils (including the upper bound for $\nun$). 
Also, see Grillakis \cite{Grillakis1990} for the analysis of the case $n(PL_+P) = n(PL_-P)$.   

Quadratic eigenvalue pencils and their spectrum are a well-studied subject with a large number of applications (see \cite{GK, tm}, and references therein). Particular areas where index theorems naturally appear are Sturm-Liouville problems \cite{BeKaTr2009} 
and gyroscopic stabilization. Gyroscopic stabilization and stability of quadratic operator pencils in general 
are related to the point spectrum of the pencil $\LL(\la) = \la^2 A + \la (D + iG) + K +iN$, where the coefficients $A, D, G, K, N$ are self-adjoint operators (see \cite{Kirillov, Merkin, Kozlov2010} and references therein) under various additional conditions for the coefficients. A survey of all important results in this area exceeds the scope of this paper and thus here we list only a few of references. 
Fundamental results for quadratic operator pencils were obtained by Krein and Langer \cite{KLa,KLb} and later extended by Adamyan and Pivovarchik \cite{AP} who also proved an index theorem similar to \eq{nLD}. 
Results that can be expressed in a form of an index theorem were obtained also by Lancaster and his coworkers \cite{LMZ2003, lz},  by Wimmer \cite{Wimmer1975}, and Chern \cite{Chern2002}  (see also reference therein).
Important index theorems for systems with dissipation $D >0$ and partial dissipation $D \ge 0$ were proved in \cite{Zajac,Wimmer1974,Kozlov1993}. Results of Zajac \cite{Zajac} that generalized the Thompson theorem for quadratic matrix pencils were later extended to operator setting by Pivovarchik \cite{Pivo1991}. 

Within the field of stability of nonlinear waves the index theorems for quadratic eigenvalue pencils are a fairly new subject. 
Chugunova and Pelinovsky \cite{ChagPel} proved the count analogous to \eq{nLD} for the quadratic Hermitian matrix pencils of the form 
$\la^2 \II + \la L + M$, where $M$ has either zero or one dimensional kernel (under a further non-degeneracy condition) via an application of the Pontryagin Invariant Subspace Theorem. Their results were reproved and extended in \cite{KollarH,KM} (see Example~\ref{ex:0}). Bronski, Johnson, and Kapitula \cite{BJK2} proved a count similar to \eq{nLD} for  the quadratic operator pencils $\LL(\la) = A + \la B + \la^2 C$,  where $A$ and $C$ are self-adjoint and $B$ is invertible skew-symmetric extending results of \cite{Pivo2007}, \cite{Shkalikov}, and \cite{Lyong}. 

Specific counts of eigenvalues for a particular class of Sturm-Liouville operators given by $J\!L$ with $J = \sign(x)$ and $L=  -d^2/dx^2 + V(x)$ in $L^2(\RR)$ were obtained in \cite{KaKoMa, BeKaTr2009}. Various types of index and eigenvalue localization theorems for definite and indefinite Sturm-Liouville problems, and particularly those that correspond to defective symmetric operators, 
can be found in \cite{BeTr2007, BeMoTr2011} and in multiple reference therein, see also Binding and Volkmer \cite{BinVol1996} where non-real spectra of $J\!L$ is studied with the use of graphical Krein signature. Further bounds particularly related to graphical Krein signature and eigenvalue branches $\mu(\la)$ (see Section~\ref{s:KM}) 
were derived in \cite{BinBrown1988}.  The local count referred to as the \emph{Krein oscillation theorem}  
was proved within the context of index theorems by Kapitula \cite{Kap} using Krein matrix theory. 
An infinitesimal version of the (local index) Theorem \ref{thm:local} for matrices is proved in \cite[Theorem 12.6]{GLR}. 

The theorem guaranteeing existence of a sequence of points in spectrum converging to zero 
for a general class of operator pencils with compact self-adjoint non-negative coefficients was proved in \cite{KollarH} by a simple homotopy argument 
as a generalization of the results of \cite{GKP}  (see also references therein). 
A homotopy argument was also used by  Maddocks and Overton \cite{MO1995} to  prove the index theorem for dissipative perturbations of  Hamiltonian systems. Index theorems within the context of isoperimetric calculus of variations were proved in \cite{GMR}.
Bronski and Johnson \cite{BJ2009} derived an index theorem for the Faddeev-Takhtajan problem by an approach analogous to work of Klaus and Shaw \cite{KS2002, KS2003} on the Zakharov-Shabat system. 
Also, Kozlov and Karapetyan \cite{KozKar2004} established the index theorem  for finite dimensional Hamiltonian systems that bounds the stable index of the system from below and connected the result to gyroscopic stabilization.

To enclose the historical review of results on index theorems let us point out that an unusually large part of the work mentioned 
within this section can be traced back to the University of Maryland at College Park, where many of the papers were written and 
many of the ideas were born. J.~H.~Maddocks, I.~Gohberg, L.~Greenberg, C.~K.~R.~T.~Jones, M.~Grillakis, R.~L.~Pego, and one of the authors of this manuscript (R.~K.) were among the others who were involved in the development of the theory.

\section{Graphical Interpretation of Index Theorems}\label{s:KM}

Within this  section we derive index theory that  encompasses Theorems~\ref{th:KKS} and \ref{th:J} and demonstrates their graphical nature. While Theorem~\ref{thm:index} was derived in \cite{KM} 
our main results contained in Theorems~\ref{thm:local} and \ref{th:table} generalize the theory developed in \cite{KM}. The analysis is for simplicity performed for matrix pencils although the results under specific assumptions can be generalized to infinitely dimensional setting \cite{BinBrown1988} in a straightforward manner.

\begin{Definitiona}
Let $\LL = \LL(\la)$ be a Hermitian matrix pencil real analytic in $\la$, and $\la_0$ its real characteristic value. 
Let $Z^{\downarrow}_{\lambda_0^-} = Z^{\downarrow}_{\lambda_0^-}(\LL)$ and 
$Z^{\downarrow}_{\lambda_0^+} = Z^{\downarrow}_{\lambda_0^+}(\LL)$
denote the number (counting multiplicity) of eigenvalue curves $\mu = \mu(\la)$ of  $\LL(\la)$ with $\mu(\la_0) = 0$ 
and $\mu(\la) < 0$ for  $\la \in (\la_0 - \eps, \la_0)$, respectively for $\la \in (\la_0, \la_0 + \eps)$, for a sufficiently small $\eps >0$. 
Similarly, let $Z^{\downarrow}_{-\infty} = Z^{\downarrow}_{-\infty}(\LL)$ 
and $Z^{\downarrow}_{+\infty}= Z^{\downarrow}_{+\infty}(\LL)$
denote the number (counting multiplicity) of eigenvalue curves $\mu = \mu(\la)$ of $\LL$ with $\mu(\la) < 0$ for $\la \in (-\infty, -K)$ and $\la \in (K, \infty)$, respectively, for a sufficiently large $K > 0$.
\end{Definitiona}
The theory applies to real analytic 
Hermitian matrix pencils $\LL(\la)$, i.e., real analytic $\LL: \RR \rightarrow \CC^{n\times n}$ for which $\LL(\la)$ is Hermitian for each $\la \in \RR$, and thus generalizes the typical case of polynomial Hermitian matrix pencils.
\begin{Theorema}
Let $\LL(\la)$ be a real analytic Hermitian matrix pencil. Then 
\begin{eqnarray}
\left( Z^{\downarrow}_{-\infty} + Z^{\downarrow}_{+\infty}  \right) - 2n(\LL(0)) 
- \left( Z^{\downarrow}_{0^+} + Z^{\downarrow}_{0^-}\right)   & = & - \sum_{\substack{\la \neq 0 \\ \la \in \sigma(\LL)}}
\sign(\la) \kappa(\la) 
\label{indexeq1}\\
\left( Z^{\downarrow}_{-\infty} - Z^{\downarrow}_{+\infty}  \right)
+ \left( Z^{\downarrow}_{0^+} -  Z^{\downarrow}_{0^-} \right)   & = &  \sum_{\substack{\la \neq 0 \\ \la \in \sigma(\LL)}} \kappa(\la)\, .
\label{indexeq2}
\end{eqnarray}
\label{thm:index}
\end{Theorema}
\begin{figure}[t]
\includegraphics[scale=0.45]{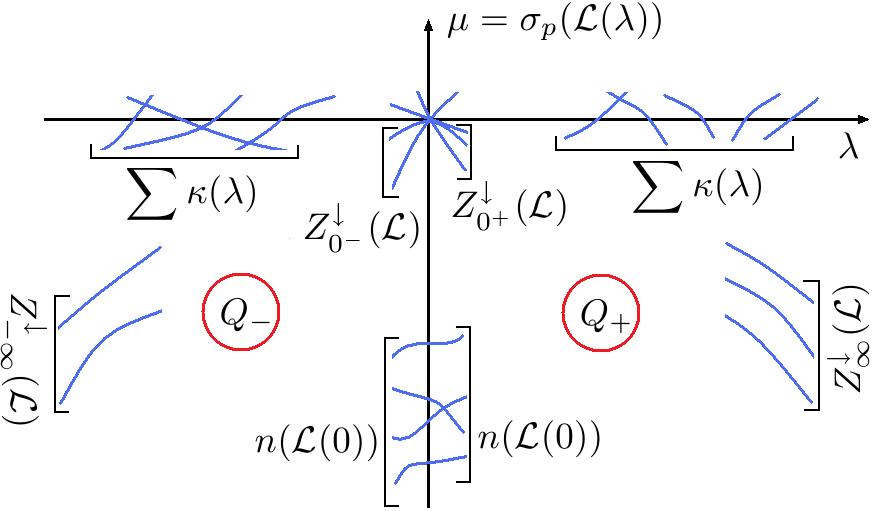}
\caption{\small{%
A schematic plot of the proof of Theorem~\ref{thm:index}. The point spectrum $\sigma_p(\LL(\la))$ is organized in eigenvalue branches $\mu = \mu(\la)$.
\label{fig2}}}
\end{figure}
\begin{proof}
Consider the (parameter dependent) eigenvalue problem \eq{Lmu}. According to the perturbation theory \cite{Kato} the eigenvalue and eigenvector curves $\mu(\la)$ 
and $u(\la)$ are analytic in $\la$. Furthermore, let $Q_\pm$ be the quadrants in the $(\la, \mu)$-plane (see Fig.~\ref{fig2}).
A simple count of curves entering and leaving $Q_\pm$ yields the counts
\begin{eqnarray}
Z^{\downarrow}_{-\infty}(\LL) - n(\LL(0)) - Z^{\downarrow}_{0^-}(\LL) - \sum_{\la < 0, \la \in \sigma(\LL)} \kappa(\la) & = & 0\, , 
\label{Q1sum}\\
Z^{\downarrow}_{+\infty}(\LL) - n(\LL(0)) - Z^{\downarrow}_{0^+}(\LL)  + \sum_{\la > 0, \la \in \sigma(\LL)} \kappa(\la)  & = & 0\, , 
\label{Q2sum}
\end{eqnarray}
Then the sum and the difference of \eq{Q1sum} and \eq{Q2sum} give \eq{indexeq1} and \eq{indexeq2}. 
\end{proof}

A local version of the index theorem can be proved analogously.
\begin{Theorema}
Let $\LL(\la)$ be a real analytic Hermitian matrix pencil. 
Then the following local index theorem holds for any real $\la_1, \la_2$ with $\la_1 < \la_2$: 
\begin{equation}
n\big(\LL(\la_1)\big) - n\big(\LL(\la_2)\big) + Z^{\downarrow}_{\la_1^+} - Z^{\downarrow}_{\la_2^-} 
= \sum_{\substack{\la_1 < \la < \la_2 \\ \la \in \sigma(\LL)}} \kappa(\la)\, .
\label{indexeq4}
\end{equation}
\label{thm:local}
\end{Theorema}

It is easy to see that 
\begin{equation}
Z^{\downarrow}_{\la_0^-} - Z^{\downarrow}_{\la_0^+} =  Z_{\la_0^+}^{\uparrow} - Z_{\la_0^-}^{\uparrow} = \kappa(\la_0)\, ,
\label{kernel0}
\end{equation}
since the eigenvalue branches vanishing at $\la_0$ at even order do not contribute  to the right hand side of \eq{kernel0}.
Then
\begin{equation}
Z^{\downarrow}_{\la_0^-} + Z^{\downarrow}_{\la_0^+} = \kappa(\la_0) + 2 Z^{\downarrow}_{\la_0^+} \, ,\qquad 
Z^{\downarrow}_{\la_0^-} + Z^{\uparrow}_{\la_0^-} =  Z^{\downarrow}_{\la_0^+} + Z^{\uparrow}_{\la_0^+}= z(\LL(\la_0))\, .
\label{kernel1}
\end{equation}

Moreover, if the matrix pencil $\LL(\la)$ has an extra structure then the terms $Z^{\downarrow}_{\pm \infty} (\LL)$ 
can be determined by the perturbation theory. Let 
\begin{equation}
\LL (\la) = L(\la) - g(\la)\II \, ,
\qquad
L(\la) = \sum_{k=0}^{p} \la^k L_k\, , \qquad
g(\la) = \sum_{k=0}^{q} \la^k g_k.
\label{pencil1}
\end{equation}
Here $L_0, \dots, L_p$ are complex Hermitian matrices $n\times n$ and $g_0, \dots g_q$ are real constants. 
There is a freedom of choice in an inclusion of identity multipliers in $g(\la)$ and $L(\la)$ but the index theorems only depend on the leading order term of $\LL(\la)$ and only differences of $g(\la)$ 
and $L(\la)$ are relevant. Therefore we ignore such an ambiguity in \eq{pencil1} and assume
$\la^p L_p \neq \la^q g_q\II$. Since $\sigma(L(\la)) \approx \la^p \sigma(L_p)$ and $g(\la) \approx \la^q g_q$ for $|\la| \rightarrow \infty$
the values of  terms $Z^{\downarrow}_{\pm \infty}$ in \eq{indexeq1} and \eq{indexeq2} 
are determined by the leading order coefficients of $g(\la)$ and $L(\la)$.

\begin{Theorema}
Let $\LL(\la)$ be a real analytic Hermitian matrix pencil of the form \eq{pencil1} and let $L_p$ be invertible. 
Then the values of the counts 
$Z^{\downarrow}_{-\infty}(\LL)$ and $Z^{\downarrow}_{+\infty}(\LL)$ 
appearing in Theorem~\ref{thm:index}
are given by the values in the table below (depending on the properties of $L(\la)$ and $g(\la)$):
\begin{center}
\begin{tabular}[h]{c c|c c}
$p, q$ & $g_q$ & $Z^{\downarrow}_{-\infty}(\LL)$ & $Z^{\downarrow}_{+\infty}(\LL)$ \\ \hline
$\mbox{$q > p$, $q$ even}$ & $g_q >0$ & $n$ & $n$ \\ \hline
$\mbox{$q > p$, $q$ even}$ & $g_q <0$ & $0$ & $0$ \\ \hline
$\mbox{$q > p$, $q$ odd}$ & $g_q >0$ & $0$ & $n$ \\ \hline
$\mbox{$q > p$, $q$ odd}$ & $g_q <0$ & $n$ & $0$ \\ \hline
$q < p$ &  & $n\big((-1)^pL_p\big)$ & $n(L_p)$ \\ \hline
$q = p$ &  & $n\big((-1)^p(L_p- q_p\II)\big)$ & $n(L_p-q_p\II)$ \\
\end{tabular}
\end{center}
\label{th:table}
\end{Theorema}
By setting $g(\la) = 0$, i.e. $q =0$, and thus $q < p$  for a non-trivial pencil $\LL(\la)$, 
in Theorems~\ref{thm:index} and \ref{th:table} one can recover generalizations of Theorems~\ref{th:KKS} 
and \ref{th:J} in a finite dimensional case \cite{KM}.
Next we illustrate how specific counts for quadratic eigenvalue pencils can be derived from
Theorems~\ref{thm:index} and \ref{th:table}.

\begin{example}\label{ex:0}
Consider the quadratic Hermitian matrix pencil  $\LL(\la)$, $\LL: \RR \rightarrow \CC^{n \times n}$
\begin{equation}
\LL(\la) = M + \la K + \la^2 \II\, , \qquad M^{\ast} = M, \quad K^{\ast} = K\, .
\end{equation}
Then, $L(\la) = M + \la K$, $p = 1$, and $g(\la) = -\la^2$, $q = 2$, and $g_q = -1$ in \eq{pencil1}. 
Therefore 
$Z^{\downarrow}_{-\infty} = Z^{\downarrow}_{+\infty} = 0$ by Theorem~\ref{th:table}. Theorem~\ref{thm:index} then gives
\begin{equation*}
2 n(M) + (Z^{\downarrow}_{0^+} + Z^{\downarrow}_{0^-})  =  
\sum_{\la \in \sigma(\LL)} \sign (\la) \kappa(\la)\, , \quad
Z^{\downarrow}_{0^+} -  Z^{\downarrow}_{0^-}  =   \sum_{\substack{\la \neq 0 \\ \la \in \sigma(\LL)}} \kappa(\la)\, .
\end{equation*}
The symmetry $(\la, \overline{\la})$ of spectrum of $\LL$ implies
\begin{equation}
2n = z(\LL) + n_r(\LL) + 2n_c(\LL) + n_i (\LL)\, ,
\label{nsum}
\end{equation}
where $n_r(\LL)$, $n_i(\LL)$ and $n_r(\LL)$ are, respectively, the numbers of real, purely imaginary, and non-real non-purely-imaginary characteristic values of $\LL(\la)$, $n_i(\LL)$ is even. Also
\begin{eqnarray*}
n_r(\LL) & = &  \sum_{\la \in \sigma(\LL), \la > 0} \left[\kappa^+(\la) + \kappa^-(\la)\right]
+  \sum_{\la \in \sigma(\LL), \la < 0} \left[\kappa^+(\la) + \kappa^-(\la) \right]\nonumber \\
& = & 
 \sum_{\la \in \sigma(\LL)} \sign(\la) \kappa(\la) 
+  2 \sum_{\la \in \sigma(\LL), \la > 0} \kappa^-(\la) + 
2 \sum_{\la \in \sigma(\LL), \la < 0}  \kappa^+(\la)\, .\ \ \ 
 \label{LM3} 
\end{eqnarray*}
Finally, we denote
\begin{equation}
n_r^-(\LL) := 2\sum_{\la \in \sigma(\LL), \la > 0} \kappa^-(\la)\, , \qquad
n_r^+(\LL) := 2\sum_{\la \in \sigma(\LL), \la > 0} \kappa^+(\la)\, , \qquad
\label{def:nrm}
\end{equation}
motivated by the case of simple real characteristic values of $\LL$ where $n_r^-$, respectively $n_r^+$, counts the number of
characteristic values of $\LL$ with the negative, respectively positive, Krein signature. 
Then \eq{nsum} can be rewritten as
\begin{equation}
2 n(M) + (Z^{\downarrow}_{0^+} + Z^{\downarrow}_{0^-})  = n_r(\LL) - n_r^-(\LL) -n_r^+(\LL) \, .
\label{LM5}
\end{equation}

{\it The symmetric case.}
If the eigenvalue problem $\LL(\la)u = \mu u$ has an additional symmetry $(\la, \mu) \rightarrow (-\la, \mu)$
the index theorem can be further simplified. 
Clearly, $ Z^{\downarrow}_{0^+} =  Z^{\downarrow}_{0^-}$, $ \kappa^-(\la) = \kappa^+(\la)$, i.e., $n_r^-(\LL) = n_r^+(\LL)$, 
and $n_r^+ (\LL) + n_r^-(\LL) = {n_r(\LL)}$.
Furthermore, all the numbers $z(\LL), n_r(\LL), n_c(\LL)$ are even.
A difference of \eq{nsum} and \eq{LM5}  
yields
\begin{equation}
\left[n - \frac{z(\LL)}{2}\right] - \left[n(M) + Z^{\downarrow}_{0^+} \right] = n_c(\LL) + \frac{n_i(\LL)}{2} 
+ n_r^-(\LL)
\label{indexLM2}
\end{equation}
The equation complementary to \eq{indexLM2} with respect to \eq{nsum} is
\begin{equation}
\left[n - \frac{z(\LL)}{2}\right] + \left[n(M) + Z^{\downarrow}_{0^+}\right]  = n_c(\LL) + \frac{n_i(\LL)}{2} 
+ n_r^+(\LL)
\label{indexLM3}
\end{equation}

In the special case of a real Hermitian $M$, a purely imaginary Hermitian $L$, and under the assumption $\Ker M \subset \Ker L$ 
the counts  \eq{indexLM2}--\eq{indexLM3} correspond to index theorems derived in \cite{ChagPel} and \cite{KollarH}.
\end{example}

\subsection{Algebraic Calculation of $Z^{\downarrow}$ and $Z^{\uparrow}$}
The counts \eq{indexeq1}, \eq{indexeq2}, and \eq{indexeq4} contain  terms $Z^{\downarrow}$, $Z^{\uparrow}$
that have a simple graphical interpretation. However, it is more traditional to express them in an algebraic form that we derive 
in this section. Theorem~\ref{Th:kernel} generalizes the relation between $Z^{\downarrow}$ and $n(D)$
in Theorem~\ref{th:KKS} (see \cite{KM} for details) that holds in the case of all Jordan blocks of the eigenvalue 0 of $J\!L$ of length two (implying 
$\dim \gKer (J\!L) = 2 \dim \Ker (L)$) to the case of the generalized kernel of $\LL(\la_0)$ of an arbitrary structure.
First, we formulate the general assumption that guarantees the required smoothness of the eigenvalue and eigenvector branches at the characteristic value of an operator pencil. 

\begin{Assumptiona}
Let $\LL= \LL(\la)$ be a real analytic self-adjoint operator pencil acting on a Hilbert space $X$ 
and let $\la_0$ be its characteristic value of a finite multiplicity. 
Let $\UU_0 = \Ker (\LL(\la_0))$ with $\dim \UU_0 = k$, and let $\eps >0$, $\delta >0$ are fixed.
Assume that for all $\la \in (\la_0-\eps, \la_0 +\eps)$
the part of the spectrum $\sigma(\LL(\la)) \cap (-\delta, \delta)$ of $\LL(\la)$ 
consists of eigenvalues organized in $C^{\infty}$ eigenvalue branches $\mu_j(\la)$, $\mu_j(\la_0) = 0$, 
$\mu_j: (\la_0-\eps, \la_0+\eps)\rightarrow  (-\delta, \delta)$, $j=1, \dots, k$, 
and that the associated eigenvector branches $u_j(\la)$, $u_j: (\la_0-\eps, \la_0+\eps) \rightarrow X$, $1\le j \le k$, are also $C^{\infty}$. 
\label{a:branches}
\end{Assumptiona}
Note that Assumption \ref{a:branches} is satisfied for Hermitian matrix pencils for any $\eps > 0$ and $\delta >0$ \cite{Kato}. 
Define for every $m\ge 0$ the sets
\begin{eqnarray}
K_m^+ & := & \{ \mu_i(\la);\  1 \le i \le k, \mu_i^{(s)}(\la_0) = 0, \ 0 \le s \le m-1 , \ \mu^{(m)}(\la_0) > 0\}, \nonumber \\
K_m^- & := & \{ \mu_i(\la);\  1 \le i \le k, \mu_i^{(s)}(\la_0) = 0, \ 0 \le s  \le m-1, \ \mu^{(m)}(\la_0) < 0\}, \nonumber \\
K_m^0 & := & \{ \mu_i(\la); \ 1 \le i \le k, \mu_i^{(s)}(\la_0) = 0, \ 0 \le s \le m\}, 
\label{setsK}
\end{eqnarray}
The sets $K_m^+$, $K_m^-$ and $K_m^0$ are disjoint for any $m \ge 0$ and
\begin{equation}
K_m^0 = K_{m+1}^- \cup K_{m+1}^+ \cup K_{m+1}^0, \qquad m\ge 0\, .
\end{equation}
For a characteristic value $\la_0$ of $\LL(\la)$ of a finite multiplicity 
$K_m^0 =\emptyset$ for $m$ large enough.
Then 
\begin{equation*}
Z^{\downarrow}_{\la_0^+}  =  \big| \bigcup_{m=1}^{\infty} K_m^- \big| = \sum_{m=1}^{\infty} |K_m^-|\, , \ 
Z^{\downarrow}_{\la_0^-}  =  \big| \bigcup_{m=1}^{\infty} K_{2m-1}^+ \cup K_{2m}^- \big| = 
\sum_{m=1}^{\infty} |K_{2m-1}^+| + |K_{2m}^-|\, .
\end{equation*}
Also, observe that $n(\LL(\la_0)) = |K_0^-|$ and that the algebraic multiplicity of the characteristic value $\la_0$ of $\LL$ is given by 
$\sum_{m=1}^{\infty} m \left( |K_m^+| + |K_m^-| \right)$.

We claim that $|K_m^\pm|$, $m \ge 0$, can be calculated as the number of positive (negative) eigenvalues of a specific matrix defined in Theorem~\ref{Th:kernel}. 
Particularly for $m$ odd, $|K_m^\pm|$ counts the number of maximal chains of root vectors of $\LL$ at $\la_0$ with positive (negative) Krein index. 
Therefore the Krein index $\kappa(\UU, \la_0)$ can be calculated by two different ways, either from the (algebraic) definition or by using the graphical Krein signature. 

\begin{example}\label{ex:1A}
Let $\LL(\la) = M + \la L + \la^2\II$ be a quadratic Hermitian matrix pencil and let 
$\UU = \{u^{[0]}, u^{[1]}, u^{[2]}\}$ be  a maximal chain of root vectors of  $\LL(\la)$ at a characteristic value $\la_0 = 0$. According to the definition of the Krein indices analogous to Definition~\ref{d:Ksig} (see \cite{KM} for details) the indices $\kappa^{\pm}(\UU,\la_0)$ count the number of positive and negative eigenvalues of the Gram matrix $W$
$$
W_{ij} = (u^{[i-1]}, Lu^{[j-1]}) +  (u^{[i-2]}, u^{[j-1]}) + (u^{[i-1]}, u^{[j-2]})\, , \quad 
i,j = 1,2,3,
$$ 
where we formally set  $u^{[-1]} = 0$. 
The characteristic polynomial $f(\la) = \det (W - \la \II)$ is a cubic polynomial with negative leading order coefficient and 
three real roots, either two of them positive and one negative or one of them positive and two negative. 
Thus 
\begin{equation}
\kappa(\UU, 0) = -\sign f(0)  = -\sign (\det W)\, .
\label{sig1}
\end{equation}
\end{example}

\begin{example}\label{ex:1B}
Consider the quadratic pencil in Example~\ref{ex:1A} and its characteristic value $\la_0=0$ with a maximal chain of root vectors
$\UU = \{u^{[0]}, u^{[1]}, u^{[2]}\}$. According to the theory \cite{KM} (see \cite{GLR} for the matrix case) there exist eigenvalue and eigenvector  branches $\mu(\la)$,  $u(\la)$ of \eq{Lmu} such that $\mu(0) = \mu'(0) = \mu''(0) = 0$ and $u(0) = u^{[0]}$. Also 
$\kappa(\UU,0) = \sign \mu'''(0) \neq 0$. For notational ease we denote 
$\LL = \LL(0)$, $u = u(0)$, $\mu = \mu(0)$, with the analogous notation for the derivatives: $\mu' = \mu'(0)$, $u' = u'(0)$, etc. We normalize 
 $(u(0), u(0)) = 1$, differentiate \eq{Lmu} three times, and take the scalar product with $u$ to obtain 
\begin{equation}
\left(u, \frac{\LL'''}{3!} u\right) + \left(u, \frac{\LL''}{2!} u'\right) + \left(u, \LL' \frac{u''}{2!}\right) + \left(u, \LL\frac{u'''}{3!}\right)= \frac{\mu'''}{3!}\, .
\label{pc2}
\end{equation}
Since $\LL$ is Hermitian the last term on the left hand side of \eq{pc2} vanishes. 
Also, differentiation of \eq{Lmu} implies  $\LL u' + \LL' u= 0$ and $\LL u'' + 2\LL' u'+ \LL'' u = 0$. 
The operator $\LL =\LL(0)$ is not invertible but $\LL + \Pi$, 
where $\Pi$ is the orthogonal projection $X \rightarrow \Ker \LL(0)$, is. 
We denote $\Lin =-( \LL + \Pi)^{-1}$ and note that $\Lin \Pi = -\Pi$. Then 
\begin{equation}
u' = \Lin \LL' u + \Pi u'\, , \quad \mbox{and} \quad
\frac{u''}{2!} =  \Lin  L' u'  + \Lin \frac{\LL''}{2!}u + \Pi \frac{u''}{2!}\, .
\label{pc23}
\end{equation}
Simple algebra (see Theorem~\ref{Th:kernel} for the derivation in the general case) reduces \eq{pc2}  to
\begin{equation}
\kappa(\UU,\la_0) = \sign \mu'''(0) = \sign \big[ 
(u, \La_3 u)  - ( \Pi u', L \Pi u') \big]\, ,
\label{pd10}
\end{equation}
where $\La_3$ is defined in \eq{Lamdef}. 
In the case of the quadratic pencil  
$$\La_3 = (M+\Pi)^{-1} L + L(M+\Pi)^{-1} + L(M+\Pi)^{-1}L(M+ \Pi)^{-1}L\, .$$
\end{example}

{\it Note.} 
The formula \eq{pd10}  has important consequences. It contains only $u = u^{[0]}$ and $u' = u'(0)$, i.e.,  it does not require  knowledge of the whole maximal chain~$\UU$, contrary to \eq{sig1}. Also, the term $( \Pi u', L \Pi u')$ is, generally, non-vanishing and 
since  $\Pi u' \in \Ker \LL$, its value is not directly encoded in a chain $\UU$ 
as  the generalized eigenvectors are determined uniquely only up to a multiple  of  $u^{[0]}$. It means that $u'(0)$ is not just an arbitrary generalized root vector to $u^{[0]}$ but it
captures an extra information $\Pi u'$ that is not, in general, contained in $u^{[1]}$. Thus the chain of root vectors $(u(0), u'(0), u''(0)/2)$ is exceptional that will be also confirmed in a general case of an arbitrary multiplicity. 
A similar calculation in the case of a quadratic matrix pencil can be found in \cite{Bosak}.

As it was illustrated in the Example~\ref{ex:1B} the graphical approach requires a proper definition of the inverse of  the operator $\LL(\la_0)$. If $\LL(\la_0)$ is Fredholm and self-adjoint then $\Ker \LL(\la_0) \perp \Ran \LL(\la_0)$. The operator $\LL(\la_0)$ 
acts on the Hilbert space $X = \Ker \LL(\la_0) \oplus \Ran \LL(\la_0)$. 
If $(v_1, v_2) \in \Ker \LL(\la_0) \oplus \Ran \LL(\la_0)$ 
then $\LL(\la_0) (v_1, v_2) = (0, \LL(\la_0) v_2)$. 
The operator $\LL(\la_0)$ is 1-to-1 on $\Ran \LL(\la_0)$ and thus the operator $\LL(\la_0) + \Pi$ is invertible as
$(\LL(\la_0) + \Pi) (v_1, v_2) = (v_1, \LL(\la_0) v_2)$. 

\begin{Definitiona}
Let $\LL = \LL(\la)$ be an operator pencil acting on a Hilbert space $X$ with a characteristic value $\la_0$ of a finite multiplicity, and let $\LL(\la_0)$ have the Fredholm index zero. 
Let $\Pi$ be an orthogonal projection $X \rightarrow \Ker \LL(\la_0)$. Then we define 
\begin{equation}
\Lin := -(\LL(\la_0)+\Pi)^{-1} \, .
\end{equation}
\end{Definitiona}
Clearly $\LL \Pi = 0$  and $\Lin \Pi = -\Pi$. Also, denote $D := d/d\la$.

{\it Notation.} 
We introduce the following notation. Let $V$ be a linear subspace of $X$ of the dimension $k$ with the orthonormal basis 
$\{v_1, \dots, v_k\}$ and let $S$ be a self-adjoint operator acting on $S$. 
First, let $V^{\ast}SV$ denote the matrix of the quadratic form 
$(\cdot, S\, \cdot)$ acting on $V$, i.e., the matrix $(v_i, S v_j)$, $1 \le i, j \le k$.
Then $\hKer(V^{\ast} S V)$ denotes 
the $k\times m$ matrix with its columns given by the $m$ column vectors that form the basis of the kernel of $V^{\ast}SV$.  
Finally, let $W := \mbox{span\,} \left[ V \hKer(V^{\ast} S V) \right]$ defines the subspace of $X$ of the dimension $m$ spanned by 
vectors obtained by the multiplication of the row vector $(v_1, \dots, v_k) \in X^{k}$ by the $m$ columns of the matrix 
$\hKer(V^{\ast} S V)$.

\begin{Theorema}
Let $\LL = \LL(\la)$ be an operator pencil on a Hilbert space $X$ with a characteristic value $\la_0$ satisfying Assumption~\ref{a:branches}, 
and let $\LL(\la_0)$ be of Fredholm index zero. Let  $K_m^\pm$, $m \ge 0$, be defined 
in \eq{setsK} and  let $U_0 = \Ker \LL(\la_0) \subset X$. 
We define recursively
\begin{equation}
U_{m+1}: = U_{m}\,  \hKer  \left( {U}_m^{\ast} H_{m+1} \, {U}_m\right) \, ,
\label{Udef}
\end{equation}
The operator $H_m$, $m \ge 1$  is defined as $H_m : = \La_m +  \De_m$.
Here 
\begin{eqnarray}
\La_m & := & \sum_{|\alpha| = m}
\frac{\LL^{(\alpha_1)}(\la_0)}{\alpha_1!} \, \Lin  \, 
\frac{\LL^{(\alpha_2)}(\la_0)}{\alpha_2!}\,  \Lin\,  \cdot  \dots \cdot\, 
\Lin \, \frac{\LL^{(\alpha_s)}(\la_0)}{\alpha_s!},\label{Lamdef} \\
\De_m & := & \sum_{|\alpha| = m} D^{\alpha_1}\Pi \, \La_{\alpha_2} \, \Pi D^{\alpha_3}
\, ,
\label{Lan}
\end{eqnarray}
where the multi-index $\alpha = (\alpha_1, \dots, \alpha_s)$ has positive integer entries and its norm is calculated as
$|\alpha| = \sum_{i=1}^s \alpha_i$. 
Then for $m \ge 1$ 
\begin{equation*}
|K_m^+| = p\, ( U_{m-1}^{\ast} \,H_m\, U_{m-1} )\, ,
\ 
|K_m^-| = n\, ( U_{m-1}^{\ast}\, H_m\, U_{m-1} )\, ,
\
U_m^{\ast} H_{m+1}\, U_{m+1} = 0\, .
\end{equation*}
\label{Th:kernel}
\end{Theorema}

\begin{proof}
We prove Theorem~\ref{Th:kernel} by mathematical induction for $m \ge 1$. Without loss of generality we set $\la_0 = 0$. 

First, let $m=1$. Let us fix $\mu_i \in K_0^0$. Then 
\begin{equation}
(\LL(\la) - \mu_i(\la)) u_i(\la) = 0 \, .
\label{pp1}
\end{equation}
Differentiation of \eq{pp1} with respect to $\la$ at $\la = \la_0$, where $\mu_i(0) = 0$
and a scalar product with $u_j$ such that $\mu_j \in K_0^0$ yields
\begin{equation}
\left(u_j, (\LL'- \mu_i') u_i \right) +(u_j,  \LL u_i') = 0\, . 
\label{pp3}
\end{equation}
where for a notational ease we drop the argument of $\LL$, $\mu$ and $u$ and their derivatives. 
The second term in \eq{pp3} vanishes as $(u_j,  \LL u_i') = (\LL u_j,  u_i') = 0$. Therefore
\begin{equation}
\left(u_j, \LL'  u_i\right) = 
\mu_i' \left(u_j, u_i \right) = \mu_i'\delta_{ij} \, , 
\label{pp4}
\end{equation}
with $\delta_{ij} = 1$ for $i = j$ and $\delta_{ij} = 0$ for $i \neq j$. 
Consequently, the matrix $\left(u_j, \LL'  u_i \right)$, $u_i, u_j \in \Ker \LL(0)$, $j = 1, \dots, k$,  
is diagonal with its eigenvalues on a diagonal. 
The number of its positive, negative and zero eigenvalues is independent of a choice of basis of $\Ker\LL (0)$ 
and it is determined by the signature of  the quadratic form $(\cdot, \LL' \cdot)$ on $\Ker \LL (0)$.
Since $\La_1 = \LL'(0)$ and $\De_1 = 0$ we derived
$$
|K_1^+| = p\, ( U_{0}^{\ast} \,H_1\, U_0 )\, ,
\qquad
|K_m^-| = n\, ( U_{0}^{\ast}\, H_1\, U_0 )\, .
$$
Let $h \in \Ker  ( U_{0}^{\ast} \,H_1\, U_0)$. Then $U_{0}^{\ast} \,H_1\, U_0 h = 0$ and 
 if $v \in U_0 \Ker  ( U_{0}^{\ast} \,H_1\, U_0) = U_1$ also $ U_{0}^{\ast} \,H_1\, U_1 = 0$.

Now assume that the statement of the theorem holds for all $j$, $j \le m$.
Let $\mu_i \in K_{m-1}^0$. Differentiation of 
\eq{pp1} $m$-times in $\la$  at $\la = 0$ together with $\mu_i(0) =\dots = \mu_i^{(m-1)}(0) = 0$  gives
\begin{equation}
\sum_{j=0}^{m} 
\frac{\LL^{(m-j)}}{(m-j)!} \,
\frac{u_i^{(j)}}{j!} = 
\frac{\mu_i^{(m)}}{m!} u_i\, .
\label{pp5}
\end{equation}
Also,
\begin{equation}
\sum_{j=0}^{s} 
\frac{\LL^{(s-j)}}{(s-j)!} \,
\frac{u_i^{(j)}}{j!} = 0
\label{pp6}
\end{equation}
for $s = 1, \dots, m-1$. Adding the term 
$\Pi u_i^{(s)} / s!$ to both sides of \eq{pp6} and inverting the operator $(\LL + \Pi)$ yields an expression for $u_i^{(s)}$. 
Now we rewrite \eq{pp5} as 
\begin{equation}
\frac{\LL^{(m)}}{(m)!} \,
u_i + 
\sum_{j=1}^{m} 
\frac{\LL^{(m-j)}}{(m-j)!} \,
\frac{u_i^{(j)}}{j!} = 
\frac{\mu_i^{(m)}}{m!} u_i\, .
\label{pp8}
\end{equation}
and express recursively each term  ${u_i^{(s)}}/{s!}$, $1 \le s  < m$, that does not contain the projection operator $\Pi$  until all the terms in the sum on the right hand side contain either $u_i$ or a projection operator $\Pi$. 
Since the total number of derivatives of $\LL$ and $u_i$ at $\la = 0$ in each term is equal $m$, and all possible decompositions of $m$ in to a sum appear, equation \eq{pp8}  reduces to
\begin{equation}
\La_m u_i - \sum_{s=1}^{m-1} \La_{m-s} \Pi \frac{u_i^{(s)}}{s!} + \LL \frac{u_i^{(m)}}{m!} = \frac{\mu_i^{(m)}}{m!} u_i\, .
\label{pp9}
\end{equation}
Taking scalar product of \eq{pp9} with  $u_j$ such that $\mu_j \in  K_{m}^0$ yields
\begin{equation}
(u_j,  \La_m u_i) - \sum_{s=1}^{m-1} (u_j, \La_{m-s}  \Pi \frac{u_i^{(s)}}{s!}) + (u_j, \LL \frac{u_i^{(m)}}{m!}) = \frac{\mu_i^{(m)}}{m!} \delta_{ij} \, .
\label{pp10}
\end{equation}
Since $\mu_j \in K_m^0$, the root vector $u_j$ can be also expressed for any $p < m$ as (compare to \eq{pp9})
\begin{equation}
\La_p u_j - \sum_{s=1}^{p-1} \La_{p-s}  \Pi \frac{u_j^{(s)}}{s!} + \LL \frac{u_j^{(p)}}{p!} =0\, .
\label{pp91}
\end{equation}
Then each individual term in the 
summand in the second term on the left-hand side of \eq{pp10} can be reduced to
\begin{eqnarray*}
(u_j, \La_{m-s}  \Pi \frac{u_i^{(s)}}{s!})  &=& (\La_{m-s} u_j,   \Pi \frac{u_i^{(s)}}{s!}) \\ 
&  = &  \sum_{r=1}^{m-s-1} (\La_{m-s-r}  \Pi \frac{u_j^{(r)}}{r!},  \Pi \frac{u_i^{(s)}}{s!}) - 
(\LL \frac{u_j^{(p)}}{p!}, \Pi \frac{u_i^{(s)}}{s!})\\
& = &  \sum_{r=1}^{m-s-1} (\La_{m-s-r}  \Pi \frac{u_j^{(r)}}{r!},  \Pi \frac{u_i^{(s)}}{s!}) 
\end{eqnarray*}
Since $ (u_j, \LL u_i^{(m)}/{m!}) = 0$ the expression in \eq{pp10} can be rewritten as 
\begin{equation}
(u_j, H_m  u_i)  = 
(u_j,  \La_m u_i) + (u_j, \De_m u_i) = \frac{\mu_i^{(m)}}{m!} \delta_{ij}\, .
\label{pp11}
\end{equation}
Therefore the matrix $U_m^{\ast} H_m U_m$ is diagonal and 
\begin{equation*}
|K_{m+1}^+|  = p\, (U_m^{\ast} H_m U_m)\, , \
|K_{m+1}^-|  = n\, (U_m^{\ast} H_m U_m)\, , \
|K_{m+1}^0|  = z\, (U_m^{\ast} H_m U_m)\, .
\end{equation*}
Multiplication of \eq{pp11} by $h \in \Ker (U_m^{\ast}  H_m U_m )$ gives
$U_m^{\ast} H_m U_m  h = 0$ that implies 
$U_m^{\ast} H_m U_{m+1}  = 0$.
\end{proof}

{\it Note.}
According to \eq{Lan} we have $\De_1 = \De_2 = 0$. 
Also,  there  are two terms $z(\LL)$ and $Z^{\downarrow}_{0^+}$ on the left hand side of \eq{indexLM3}
that are connected with the properties of (generalized) kernel of $\LL$. 
It is easy to see that under an assumption $\Ker(M) \subset \Ker(L)$ one has
$z(\LL) = 2 \dim (\Ker(M))$ that leads to the simplified expression in \eq{nLD} (see \cite{KollarH} for the details).

\section{Conclusions}
We presented a unifying view of the index theorems frequently used across various fields. Furthermore, we demonstrated a special property of the chain of root vectors generated by the graphical method that allowed us to derive formulae for the number of 
eigenvalue curves of the eigenvalue problem $\LL(\la)u = \mu u$ entering the lower half-plane of the plane $(\mu, \la)$ through the 
characteristic value $\la_0$ of $\LL(\la)$ that did not require knowledge of the full chain of the root vectors. 
Both these results demonstrate the extraordinary beauty and power of the graphical approach. 
Let us conclude by a quote from \cite{BinVol1996}: \emph{``Eigencurves} (produced by the graphical approach)
\emph{seem to provide a very useful tool in a variety of circumstances, and their theory and applications are quite underdeveloped"}, 
a statement that certainly remains true even today. 

\section{Acknowledgment}
R.~K.~would like to thank the European Commission for the financial support via the Marie Curie International Reintegration Grant 239429.

\end{document}